\documentclass[12pt]{amsart}
\usepackage[top=35mm, bottom=35mm, left=25mm, right=25mm]{geometry}
\usepackage{xcolor}
\usepackage{bookmark}
\usepackage[all]{xy}  
\usepackage{graphicx}
\usepackage{amssymb,amsmath,amsthm, amsfonts}
\usepackage{mathtools}

\usepackage{fancyhdr}
\newtheorem{thm}{Theorem}[section]
\newtheorem{lem}[thm]{Lemma}

\newtheorem{cor}[thm]{Corollary}

\theoremstyle{definition}

\newtheorem*{rem*}{Remark}

\newcommand{\R}{\mathbb R}

\numberwithin{equation}{section}

\begin{document}
\title{Gradient Estimate for solutions of $\Delta v+v^r-v^s= 0$ on a complete Riemannian Manifold}

\author{Youde Wang}
\address{1. School of Mathematics and Information Sciences, Guangzhou University; 2. Hua Loo-Keng Key Laboratory
of Mathematics, Institute of Mathematics, Academy of Mathematics and Systems Science, Chinese Academy
of Sciences, Beijing 100190, China; 3. School of Mathematical Sciences, University of Chinese Academy of Sciences,
Beijing 100049, China.}
\email{wyd@math.ac.cn}

\author{Aiqi Zhang}
\address{School of Mathematics and Information Sciences, Guangzhou University}
\email{zhangaiqi@gzdx.wecom.work}

\date{\today}

\begin{abstract}
In this paper we consider the gradient estimates on positive solutions to the following elliptic equation defined on a complete Riemannian manifold $(M,\,g)$: $$\Delta v+v^r-v^s= 0,$$
where $r$ and $s$ are two real constants.

When$(M,\,g)$ satisfies $Ric \geq -(n-1)\kappa$ (where $n\geq2$ is the dimension of $M$ and $\kappa$ is a nonnegative constant), we employ the Nash-Moser iteration technique to derive a Cheng-Yau's type gradient estimate for positive solution to the above equation under some suitable geometric and analysis conditions. Moreover, it is shown that when the Ricci curvature of $M$ is nonnegative, this elliptic equation does not admit any positive solution except for $u\equiv 1$ if $r<s$ and $$1<r<\frac{n+3}{n-1}\quad\quad ~~\mbox{or}~~\quad 1<s<\frac{n+3}{n-1}.$$

\end{abstract}

\maketitle

\section{Introduction}
In the last half century the following semi-linear elliptic equation defined on $\mathbb{R}^n$
$$\Delta u + h(x, u)=0,$$
where $h: \mathbb{R}^n\times\mathbb{R}\to\mathbb{R}$ is a continuous or smooth function, attracted many mathematicians to pay attention to the study on the existence, sigularity and various symmetries of its solutions. For instance, Caffarelli, Gidas and Spruck in \cite{CGS} discussed the solutions to some special form of the above equation, written by
$$\Delta u + g(u)=0,$$
with an isolated singularity at the origin, and studied non-negative smooth solutions of the conformal invariant equation
\begin{equation*}
\Delta u+u^{(n+2)/(n-2)}=0,
\end{equation*}
where $n\geq3$. Later, W.-X. Chen and C.-M. Li classified the solutions to the following equation
$$\Delta u+u^r=0$$ in the critical or subcritical case in \cite{CL}, and C.-M. Li \cite{Li} simplified and further exploited the ``measure theoretic" variation, introduced in \cite{CGS}, of the method of moving planes.

On the other hand, one also studied the following prescribed scalar curvature equation (PSE) on $\R^n$
\begin{equation*}
\Delta u + Ku^{(n+2)/(n-2)}=0,
\end{equation*}
where $n\geq3$ and $K: \R^n\to\R$ is a smooth function. It was shown by Ni in \cite{N} that if $K$ is bounded and $|K|$ decays in a three dimensional subspace faster than $C/|x|^2$ at $\infty$ for some constant $C>0$, then the above equation has infinitely many bounded solutions in $\R^n$. It was also shown in \cite{N} that if $K$ is negative and decays slower than $-C/|x|^2$ at $\infty$, then the PSE equation has no positive solutions on $\R^n$. Later, Lin (cf. \cite{lf}) improved this result to the case when $K\leq-C/|x|^2$ at $\infty$. This gives an essentially complete picture for the negative $K$ case. When $K\geq0$, the situation is much more complicated. It was proved in \cite{N} that if $K\geq C|x|^2$, then the PSE equation admits no positive solutions on $\R^n$. For the case $K$ is bounded, this problem was studied by Ding and Ni in \cite{DN1} by using ``finite domain approximation". In particular, they proved the following,
\begin{thm}[\cite{DN1}]\label{dn}
For any $b>0$, the equation $$\Delta u + au^{r}=0$$ on $\R^n$, where $r\geq (n+2)/(n-2)$ and $a$ is positive constant, possesses a positive solution $v$ with $\|v\|_{L^{\infty}}=b$.
\end{thm}

Very recently, Y.-D. Wang and G.-D. Wei \cite{WW} adopted the Nash-Moser iteration to study the nonexistence of positive solutions to the above Lane-Emden equation with a positive constant $a$, i.e.,
$$\Delta u+au^{r}=0$$
defined on a noncompact complete Riemannian manifold $(M, g)$ with $\dim(M)=n \geq 3$, and improve some results in \cite{PWW}. Later, inspired by the work of X.-D. Wang and L. Zhang \cite{WZ}, J. He, Y.-D. Wang and G.-D. Wei \cite{He-Wang-Wei} also discussed the gradient estimates and Liouville type theorems for the positive solution to the following generalized Lane-Emden equation
$$\Delta_pu + au^r=0.$$
Especially, the results obtained in \cite{WW} are also improved. It is shown in \cite{He-Wang-Wei} that, if the Ricci curvature of underlying manifold is nonnegative and
$$r\in\left(-\infty,\, (n+3)/(n-1)\right)$$
the above equation with $a>0$ and $p=2$ does not admit any positive solution.

It is worthy to point out that the case $a<0$ is also discussed in \cite{He-Wang-Wei}. Inspired by \cite{He-Wang-Wei}, one would like to ask naturally what happens if the nonlinear term in Lane-Emden equation is replaced by $v^r-v^s$. More precisely, one would like to know how $r$ and $s$ affect each other. So, in this paper we focus on the following elliptic equation defined on a complete Riemannian manifold $(M,\,g)$:
\begin{equation}\label{eqno1.1}
\Delta v+v^r-v^s= 0,
\end{equation}
where $r$ and $s$ are real constants, $\Delta$ is the Laplace-Beltrami operator on $(M,\,g)$ with respect to the metric $g$. In other words, here $h(x, u)$ is of the special form $h(x, u)\equiv u^r-u^s$. In fact, in the case $r=1$ and $s=3$ this equation is just the well-known Allen-Cahn equation (see \cite{DWY})
$$\Delta v + v(1-v^2)=0.$$

It is worthy to point out that the method adopted here can be used to deal with the general equation $\Delta u + h(x, u)=0$ under some technical conditions and we will discuss the general equation in a forthcoming paper.

In the sequel, we always let $(M,\,g)$ be a complete Riemannian manifold with Ricci curvature $Ric \geq -(n-1)\kappa$. For the sake of convenience, we need to make some conventions firstly. Throughout this paper, unless otherwise mentioned, we always assume $\kappa\geq0$, $n\geq 2$ is the dimension of $M$, $r$ and $s$ are two real constants. Moreover, we denote $B_{R}=B(o,\,R)$ for any $R>0$.

Now, we are ready to state our results.

\begin{thm}
\label{thm1}
Let $(M,\,g)$ be a complete Riemannian manifold with $Ric \geq -(n-1)\kappa$ and $\dim(M)\geq 2$. Assume that $v$ is a smooth positive solution of (\ref{eqno1.1}) on the geodesic ball $B_R\subset M$.
If $r\leq s$ and $1<r<\frac{n+3}{n-1}$, or $r\leq s$ and $1<s<\frac{n+3}{n-1}$,
then we have:
\begin{equation}\label{eqno1.2}
\frac{\left|\nabla v\right|^2}{v^2}\leq c(n,\,r,\,s)\frac{\left(1+\sqrt\kappa R\right)^2}{R^2}\quad\mbox{on}~B_{R/2}.
\end{equation}
\end{thm}
Immediately, we have the following direct corollary:

\begin{cor}\label{cor1}
Let $(M,\,g)$ be a noncompact complete Riemannian manifold with nonnegative Ricci curvature and $\dim(M)\geq 2$. The equation (\ref{eqno1.1}) admits a unique positive solution $v\equiv 1$ if $r<s$ and
$$1<r<\frac{n+3}{n-1}\quad\quad ~~\mbox{or}~~\quad 1<s<\frac{n+3}{n-1}.$$
\end{cor}

Moreover, according to the above corollary we have the following conclusion:
\begin{cor}\label{cor2}
Let $(M,\,g)$ be a noncompact complete Riemannian manifold with nonnegative Ricci curvature and $\dim(M)=2$. Then, the Allen-Cahn equation on $(M, g)$ does not admit any positive solution except for $v\equiv 1$.
\end{cor}

It is worthy to point out that our method is useful for the equation \eqref{eqno1.1} on n-dimensional complete Riemannian manifolds for any $n\geq2$. We will firstly discuss the case $n\geq3$ concretely, then briefly discuss the case $n=2$.

The rest of this paper is  organized as follows. In Section 2, we will give a detailed estimate of Laplacian of $$|\nabla\ln v|^2=\dfrac{|\nabla v|^2}{v^2},$$ where $v$ is the positive solution of the equation (\ref{eqno1.1}) with $r$ and $s$ satisfying the conditions we set in the Theorem \ref{thm1}. Then we need to recall the Saloff-Coste's Sobolev embedding theorem. In Section 3, we use the Moser iteration to prove Theorem \ref{thm1} in the case $n\geq3$, then we briefly discuss the case $n=2$ using the same approach, and finally we give the proof of Corollary \ref{cor1}.

\section{Prelimanary}
Let $v$ be a positive smooth solution to the elliptic equation:
\begin{equation}\label{eqno2.1}
\Delta v+v^r-v^s= 0\quad\mbox{on}~ B_R,
\end{equation}
where $r$ and $s$ are two real constants. Set $u=-\ln v$. We compute directly and obtain
$$
\Delta u=|\nabla u|^2+e^{(1-r)u}-e^{(1-s)u}.
$$
For convenience, we denote $h=|\nabla u|^2$. By a direct calculation we can verify
\begin{equation}
\label{eqno2.2}
\Delta u=h+e^{(1-r)u}-e^{(1-s)u}.
\end{equation}
\begin{lem}\label{lem2.1}
Let $h=|\nabla u|^2$ and $u=-\ln v$ where $v$ is a positive solution to \eqref{eqno1.1}. If $r\leq s$ and
$$1<r<\frac{n+3}{n-1}\quad\quad ~~\mbox{or}~~\quad 1<s<\frac{n+3}{n-1},$$
then at the point where $h\neq0$, there exist $\tilde\alpha\in\left[1,\,+\infty\right)$ and $\tilde\rho\in\left(0,\,\frac{2}{n-1}\right]$ such that
\begin{equation}\label{eqno2.3}
\frac{\Delta \left(h^{\tilde\alpha}\right)}{\tilde\alpha h^{\tilde\alpha-1}}\geq\tilde\rho h^2-2(n-1)\kappa h+\frac{2(n-2)\langle\nabla u,\,\nabla h\rangle}{n-1},
\end{equation}
where $\tilde\alpha=\tilde\alpha(n,\,r,\,s)$ and $\tilde\rho=\tilde\rho(n,\,r,\,s)$ are depend on $n,\,r$ and $s$. \end{lem}

\begin{proof}
At the point $h\neq0$, firstly, by Bochner formula we have
\begin{equation}\label{eqno2.4}
\begin{aligned}
\Delta h
= & 2|\nabla ^2 u|^2 + 2 Ric(\nabla u,\,\nabla u)+ 2\langle \nabla u,\,\nabla \Delta u\rangle\\
\geq& 2|\nabla ^2 u|^2 -2(n-1)\kappa h+ 2\langle \nabla u,\,\nabla h\rangle + 2h \left[(1-r)e^{(1-r)u}-(1-s)e^{(1-s)u} \right].
\end{aligned}
\end{equation}

Now, we choose a suitable local orthonormal frame $\left\{\xi_i \right\}_{i=1}^{n}$ such that $\nabla u=|\nabla u|\xi_1$. If
we denote $\nabla u=\sum_{i=1}^{n}u_i\xi_i$, it is easy to see
$$
u_1=|\nabla u|\quad\mbox{and}\quad u_i=0
$$
for any $2 \leq i \leq n$. Noticing that
$$
\sum_{i=2}^{n} u_{ii}=h+e^{(1-r)u}-e^{(1-s)u}-u_{11},
$$
we have
$$
\begin{aligned}
|\nabla ^2 u|^2
\geq& u_{11}^{2}+\sum_{i=2}^{n} u_{ii}^2\\
\geq& u_{11}^{2}+\frac{1}{n-1} \left( h+e^{(1-r)u}-e^{(1-s)u}-u_{11}\right)^2\\
=&\frac{h^2}{n-1}+\frac{n u_{11}^{2}}{n-1}+\frac{\left(e^{(1-r)u}-e^{(1-s)u}\right)^2}{n-1}-\frac{2hu_{11}}{n-1}\\
&-2u_{11}\frac{e^{(1-r)u}-e^{(1-s)u}}{n-1}+2h\frac{e^{(1-r)u}-e^{(1-s)u}}{n-1}.
\end{aligned}
$$
Since
$$
\begin{aligned}
2h u_{11}
=&2|\nabla u|^2\nabla^2u(\xi_1,\,\xi_1)\\
=&2|\nabla u|^2\langle\nabla_{\xi_1}du, \xi_1\rangle\\
=&2|\nabla u|^2\left[\xi_1\left(|\nabla u|\right)-\left(\nabla_{\xi_1}\xi_1\right)u\right]\\
=&2|\nabla u|^2\langle \xi_1,\,\nabla|\nabla u|\rangle\\
=&\langle |\nabla u|\xi_1,\,2|\nabla u|\nabla|\nabla u|\rangle\\
=&\langle\nabla u,\,\nabla h\rangle,
\end{aligned}
$$
we have
$$
\begin{aligned}
|\nabla ^2 u|^2
\geq&\frac{h^2}{n-1}+\frac{n u_{11}^{2}}{n-1}+\frac{\left(e^{(1-r)u}-e^{(1-s)u}\right)^2}{n-1}-\frac{\langle\nabla u,\,\nabla h\rangle}{n-1}\\
& -2u_{11}\frac{e^{(1-r)u}-e^{(1-s)u}}{n-1}+2h\frac{e^{(1-r)u}-e^{(1-s)u}}{n-1}.
\end{aligned}
$$
Hence, by substituting the above inequality into (\ref{eqno2.4}) we get
$$
\begin{aligned}
\Delta h
\geq & \frac{2h^2}{n-1}+\frac{2n u_{11}^{2}}{n-1}+\frac{2\left(e^{(1-r)u}-e^{(1-s)u}\right)^2}{n-1}-\frac{2\langle\nabla u,\,\nabla h\rangle}{n-1} -4u_{11}\frac{e^{(1-r)u}-e^{(1-s)u}}{n-1}\\
&+4h\frac{e^{(1-r)u}-e^{(1-s)u}}{n-1}-2(n-1)\kappa h+ 2\langle \nabla u,\,\nabla h\rangle + 2h  \left[(1-r)e^{(1-r)u}-(1-s)e^{(1-s)u} \right]\\
=&\frac{2h^2}{n-1}+\frac{2n u_{11}^{2}}{n-1}+\frac{2\left(e^{(1-r)u}-e^{(1-s)u}\right)^2}{n-1}-4u_{11}\frac{e^{(1-r)u}-e^{(1-s)u}}{n-1}+4h\frac{e^{(1-r)u}-e^{(1-s)u}}{n-1}\\
& + 2h \left[(1-r)e^{(1-r)u}-(1-s)e^{(1-s)u} \right]-2(n-1)\kappa h+\frac{2(n-2)\langle\nabla u,\,\nabla h\rangle}{n-1}.
\end{aligned}
$$
For any $\alpha\geq1$, we have
$$
\Delta \left(h^\alpha\right)=\alpha\left(\alpha-1\right)h^{\alpha-2}|\nabla h|^2+\alpha h^{\alpha-1}\Delta h,
$$
therefore,
\begin{equation}\label{eqno2.5}
\frac{\Delta \left(h^\alpha\right)}{\alpha h^{\alpha-1}}=(\alpha-1)h^{-1}|\nabla h|^2+\Delta h.
\end{equation}
Since
$$
\begin{aligned}
|\nabla h|^2
=&\sum_{i=1}^{n}|2u_1u_{1i}|^2\\
=&4h\sum_{i=1}^{n}u_{1i}^2\\
\geq&4hu_{11}^2,
\end{aligned}
$$
we can see that
\begin{equation}\label{eqno2.6}
\begin{aligned}
\frac{\Delta \left(h^\alpha\right)}{\alpha h^{\alpha-1}}
\geq&\frac{2h^2}{n-1}+\left[4(\alpha-1)+\frac{2n}{n-1}\right]u_{11}^{2}+\frac{2\left(e^{(1-r)u}-e^{(1-s)u}\right)^2}{n-1}\\
&-4u_{11}\frac{e^{(1-r)u}-e^{(1-s)u}}{n-1}+4h\frac{e^{(1-r)u}-e^{(1-s)u}}{n-1}\\
&+2h \left[(1-r)e^{(1-r)u}-(1-s)e^{(1-s)u} \right]\\
&-2(n-1)\kappa h+\frac{2(n-2)\langle\nabla u,\,\nabla h\rangle}{n-1}.
\end{aligned}
\end{equation}
Since
$$
\begin{aligned}
&\left[4(\alpha-1)+\frac{2n}{n-1}\right]u_{11}^{2}+\frac{2\left(e^{(1-r)u}-e^{(1-s)u}\right)^2}{n-1}-4u_{11}\frac{e^{(1-r)u}-e^{(1-s)u}}{n-1}\\
=&\frac{2\left(e^{(1-r)u}-e^{(1-s)u}\right)^2}{n-1}+\frac{4(\alpha-1)(n-1)+2n}{n-1}\left[u_{11}-\frac{e^{(1-r)u}-e^{(1-s)u}}{2(\alpha-1)(n-1)+n}\right]^2\\
&-\frac{4(\alpha-1)(n-1)+2n}{n-1}\left[\frac{e^{(1-r)u}-e^{(1-s)u}}{2(\alpha-1)(n-1)+n}\right]^2\\
\geq&\frac{2\left(e^{(1-r)u}-e^{(1-s)u}\right)^2}{n-1}-\frac{4(\alpha-1)(n-1)+2n}{n-1}\left[\frac{e^{(1-r)u}-e^{(1-s)u}}{2(\alpha-1)(n-1)+n}\right]^2\\
=&\left[\frac{2}{n-1}-\frac{1}{2(\alpha-1)(n-1)+n }\cdot\frac{2}{(n-1)}\right]\left(e^{(1-r)u}-e^{(1-s)u}\right)^2\\
=&\frac{2}{n-1}\frac{2(\alpha-1)(n-1)+n -1}{2(\alpha-1)(n-1)+n }\left(e^{(1-r)u}-e^{(1-s)u}\right)^2\\
=&\frac{2(2\alpha-1)}{2(\alpha-1)(n-1)+n}\left(e^{(1-r)u}-e^{(1-s)u}\right)^2,
\end{aligned}
$$
then, it follows
$$
\begin{aligned}
\frac{\Delta \left(h^\alpha\right)}{\alpha h^{\alpha-1}}
\geq&\frac{2h^2}{n-1}+\frac{2(2\alpha-1)}{2(\alpha-1)(n-1)+n}\left(e^{(1-r)u}-e^{(1-s)u}\right)^2+4h\frac{e^{(1-r)u}-e^{(1-s)u}}{n-1}\\
&+2h \left[(1-r)e^{(1-r)u}-(1-s)e^{(1-s)u} \right]-2(n-1)\kappa h+\frac{2(n-2)\langle\nabla u,\,\nabla h\rangle}{n-1}\\
=&\frac{2h^2}{n-1}+\frac{2(2\alpha-1)}{2(\alpha-1)(n-1)+n}\left(e^{(1-r)u}-e^{(1-s)u}\right)^2+2h\left(\frac{n+1}{n-1}-r\right)e^{(1-r)u}\\
&-2h\left(\frac{n+1}{n-1}-s\right)e^{(1-s)u}-2(n-1)\kappa h+\frac{2(n-2)\langle\nabla u,\,\nabla h\rangle}{n-1}.
\end{aligned}
$$
If $r\leq s$, from the above inequality we can see that there holds
\begin{equation}\label{add}
\begin{aligned}
\frac{\Delta \left(h^\alpha\right)}{\alpha h^{\alpha-1}}
\geq&\frac{2h^2}{n-1}+\frac{2(2\alpha-1)}{2(\alpha-1)(n-1)+n}\left(e^{(1-r)u}-e^{(1-s)u}\right)^2+2h\left(\frac{n+1}{n-1}-r\right)e^{(1-r)u}\\
&-2h\left(\frac{n+1}{n-1}-r\right)e^{(1-s)u}-2(n-1)\kappa h+\frac{2(n-2)\langle\nabla u,\,\nabla h\rangle}{n-1}\\
=&\frac{2h^2}{n-1}+\frac{2(2\alpha-1)}{2(\alpha-1)(n-1)+n}\left(e^{(1-r)u}-e^{(1-s)u}\right)^2\\
&+2h\left(\frac{n+1}{n-1}-r\right)\left(e^{(1-r)u}-e^{(1-s)u}\right)-2(n-1)\kappa h+\frac{2(n-2)\langle\nabla u,\,\nabla h\rangle}{n-1},
\end{aligned}
\end{equation}
on the other hand, there also holds true
\begin{equation}\label{add*}
\begin{aligned}
\frac{\Delta \left(h^\alpha\right)}{\alpha h^{\alpha-1}}
\geq&\frac{2h^2}{n-1}+\frac{2(2\alpha-1)}{2(\alpha-1)(n-1)+n}\left(e^{(1-r)u}-e^{(1-s)u}\right)^2+2h\left(\frac{n+1}{n-1}-s\right)e^{(1-r)u}\\
&-2h\left(\frac{n+1}{n-1}-s\right)e^{(1-s)u}-2(n-1)\kappa h+\frac{2(n-2)\langle\nabla u,\,\nabla h\rangle}{n-1}\\
=&\frac{2h^2}{n-1}+\frac{2(2\alpha-1)}{2(\alpha-1)(n-1)+n}\left(e^{(1-r)u}-e^{(1-s)u}\right)^2\\
&+2h\left(\frac{n+1}{n-1}-s\right)\left(e^{(1-r)u}-e^{(1-s)u}\right)-2(n-1)\kappa h+\frac{2(n-2)\langle\nabla u,\,\nabla h\rangle}{n-1}.
\end{aligned}
\end{equation}
Since
$$\begin{aligned}
&\frac{2(2\alpha-1)}{2(\alpha-1)(n-1)+n}\left(e^{(1-r)u}-e^{(1-s)u}\right)^2+2h\left(\frac{n+1}{n-1}-r\right)\left(e^{(1-r)u}-e^{(1-s)u}\right)\\
=&\frac{2(2\alpha-1)}{2(\alpha-1)(n-1)+n}\left[\left(e^{(1-r)u}-e^{(1-s)u}\right)^2+h\left(\frac{n+1}{n-1}-r\right)\frac{2(\alpha-1)(n-1)+n}{2(2\alpha-1)}\right]^2\\
&-\frac{2(\alpha-1)(n-1)+n}{2(2\alpha-1)}\left(\frac{n+1}{n-1}-r\right)^2 h^2\\
\geq&-\frac{2(\alpha-1)(n-1)+n}{2(2\alpha-1)}\left(\frac{n+1}{n-1}-r\right)^2 h^2
\end{aligned}$$
and
$$\begin{aligned}
&\frac{2(2\alpha-1)}{2(\alpha-1)(n-1)+n}\left(e^{(1-r)u}-e^{(1-s)u}\right)^2+2h\left(\frac{n+1}{n-1}-s\right)\left(e^{(1-r)u}-e^{(1-s)u}\right)\\
=&\frac{2(2\alpha-1)}{2(\alpha-1)(n-1)+n}\left[\left(e^{(1-r)u}-e^{(1-s)u}\right)^2+h\left(\frac{n+1}{n-1}-s\right)\frac{2(\alpha-1)(n-1)+n}{2(2\alpha-1)}\right]^2\\
&-\frac{2(\alpha-1)(n-1)+n}{2(2\alpha-1)}\left(\frac{n+1}{n-1}-s\right)^2 h^2\\
\geq&-\frac{2(\alpha-1)(n-1)+n}{2(2\alpha-1)}\left(\frac{n+1}{n-1}-s\right)^2 h^2,
\end{aligned}$$
we substitute the above two inequalities into \eqref{add} and \eqref{add*} respectively to obtain
\begin{equation}\label{eqno2.7}
\begin{aligned}
\frac{\Delta \left(h^\alpha\right)}{\alpha h^{\alpha-1}}
\geq&\left[\frac{2}{n-1}-\frac{2(\alpha-1)(n-1)+n}{2(2\alpha-1)}\left(\frac{n+1}{n-1}-r\right)^2\right]h^2\\
&-2(n-1)\kappa h+\frac{2(n-2)\langle\nabla u,\,\nabla h\rangle}{n-1}
\end{aligned}
\end{equation}
and
\begin{equation}\label{eqno2.8}
\begin{aligned}
\frac{\Delta \left(h^\alpha\right)}{\alpha h^{\alpha-1}}
\geq&\left[\frac{2}{n-1}-\frac{2(\alpha-1)(n-1)+n}{2(2\alpha-1)}\left(\frac{n+1}{n-1}-s\right)^2\right]h^2\\
&-2(n-1)\kappa h+\frac{2(n-2)\langle\nabla u,\,\nabla h\rangle}{n-1}.
\end{aligned}
\end{equation}
\medskip

Next, we need to discuss the following two cases respectively:

{\bf Case 1}:      $$r\leq s\quad\quad~~\mbox{and}~~\quad 1<r<\frac{n+3}{n-1}.$$

For this case, we focus on (\ref{eqno2.7}). In the present situation, we have
$$ 0\leq \left|\frac{n+1}{n-1}-r\right|<\frac{2}{n-1}\quad\quad\mbox{and}\quad\quad 0\leq\left(\frac{n+1}{n-1}-r\right)^2<\frac{4}{(n-1)^2}.$$
Let
$$k(\alpha)=\frac{2(\alpha-1)(n-1)+n}{2(2\alpha-1)}.$$
Since $k(\alpha)$ is a monotone decreasing function with respect to $\alpha$ on $[1,\,\infty)$ and
$$
\lim\limits_{\alpha\to+\infty}\frac{2(\alpha-1)(n-1)+n}{2(2\alpha-1)}=\frac{n-1}{2},
$$
we obtain that for any $\alpha\geq1$ there holds true
$$
\frac{n-1}{2}\leq k(\alpha)\leq\frac{n}{2}.
$$
Let $$t=\left(\frac{n+1}{n-1}-r\right)^2\in\left[0,\,\frac{4}{(n-1)^2}\right).$$
It is easy to see that
$$k(\alpha)t<\frac{n}{2}\cdot\frac{4}{n(n-1)}=\frac{2}{n-1},$$
as $t\in\left[0,\,\frac{4}{n(n-1)}\right)$.
For any $t\in\left[\frac{4}{n(n-1)},\,\frac{4}{(n-1)^2}\right)$, it is not difficult to find that there exists some $k_t\in \left(\frac{n-1}{2},\,\frac{n}{2}\right]$ such that
$$tk_t=\frac{2}{n-1}.$$
Therefore, for any $k(\alpha)\in\left(\frac{n-1}{2},\,k_t\right)$, we have $$k(\alpha)t<\frac{2}{n-1}.$$
Hence, for any given $1< r< \frac{n+3}{n-1}$, we can choose $\alpha=\alpha_{n,\,r}$ large enough such that
$$k(\alpha_{n,\,r})<k_t$$
where $t=\left(\frac{n+1}{n-1}-r\right)^2$. Hence, for such $\alpha=\alpha_{n,\,r}$ we have
$$
\frac{2(\alpha-1)(n-1)+n}{2(2\alpha-1)}\left(\frac{n+1}{n-1}-r\right)^2<\frac{2}{n-1}
$$
and
$$
\rho(n,\,r,\,\alpha)=\frac{2}{n-1}-\frac{2(\alpha-1)(n-1)+n}{2(2\alpha-1)}\left(\frac{n+1}{n-1}-r\right)^2>0.
$$
\medskip

{\bf Case 2}:      $$r\leq s\quad\quad~~\mbox{and}~~\quad 1<s<\frac{n+3}{n-1}.$$

For this case, we focus on (\ref{eqno2.8}), and similarly we have
$$ 0\leq \left|\frac{n+1}{n-1}-s\right|<\frac{2}{n-1}\quad\quad\mbox{and}\quad\quad 0\leq\left(\frac{n+1}{n-1}-s\right)^2<\frac{4}{(n-1)^2}.$$
Hence, by the same way as in the case 1 we can see that there exists $\alpha=\alpha_{n,\,s}$ large enough such that, for any $1< s< \frac{n+3}{n-1}$, there holds true
$$
\frac{2(\alpha-1)(n-1)+n}{2(2\alpha-1)}\left(\frac{n+1}{n-1}-s\right)^2<\frac{2}{n-1}
$$
and
$$
\rho(n,\,s,\,\alpha)=\frac{2}{n-1}-\frac{2(\alpha-1)(n-1)+n}{2(2\alpha-1)}\left(\frac{n+1}{n-1}-s\right)^2>0.
$$

Based on the above argument, now we let
$$\tilde\alpha=\tilde\alpha(n,\,r,\,s)=
\begin{cases}
\alpha_{n,\,r},\,\quad r\leq s ~~\mbox{and}~~1<r<\frac{n+3}{n-1},\\
\alpha_{n,\,s},\,\quad r\leq s ~~\mbox{and}~~1<s<\frac{n+3}{n-1},
\end{cases}$$
and
$$\tilde\rho=\tilde\rho(n,\,r,\,s)=
\begin{cases}
\rho(n,\,r,\,\alpha_{n,\,r}),\,\quad r\leq s ~~\mbox{and}~~1<r<\frac{n+3}{n-1},\\
\rho(n,\,s,\,\alpha_{n,\,s}),\,\quad r\leq s ~~\mbox{and}~~1<s<\frac{n+3}{n-1}.
\end{cases}$$
Obviously, we obtain the required (\ref{eqno2.3}). Thus, the proof of Lemma \ref{lem2.1} is completed.\end{proof}

Next, we need to recall the Saloff-Coste's Sobolev embedding theorem (Theorem 3.1 in \cite{L. Saloff-Coste1992}), which plays a key role on the arguments (Moser iteration) taken here.

\begin{thm}\label{thm Sobolev}
(the Saloff-Coste's Sobolev embedding theorem)
Let $(M,\,g)$ be a complete Riemannian manifold with $Ric\geq -(n-1)\kappa $.
For any $n>2$, there exist a constant $c_n$, depending only on $n$, such that for all $B\subset M$ we have
$$
\left(\int_{B}h^{2\chi }\right)^{\frac{1}{\chi }}\leq e^{c_n\left(1+\sqrt \kappa R\right)}V^{-\frac{2}{n}}R^2\left(\int_{B}\left|\nabla h\right|^2+\int_{B}R^{-2}h^2\right),
\quad h\in C_{0}^{\infty}(B),
$$
where $R$ and $V$ are the radius and volume of $B$, constant $\chi =\frac{n}{n-2}$.
For $n=2$, the above inequality holds with $n$ replaced by any fixed $n'>2$.
\end{thm}

\section{Proof of main results}
In this section we first provide the proof of Theorem \ref{thm1} and need to discuss two cases, i.e., the case $n\geq 3$ and the case $n=2$. After that, we will give the proof of Corollary \ref{cor1}.

\subsection{The case $n\geq 3$}
We first focus on the proof of Theorem \ref{thm1} in the case $n\geq3$.
Throughout this subsection, unless otherwise mentioned, $n\ge3$, $r$ and $s$ are two real constants which satisfy that $r\leq s$ and
$$1<r<\frac{n+3}{n-1}\quad\quad ~~\mbox{or}~~\quad 1<s<\frac{n+3}{n-1}.$$

\begin{lem}\label{lem3.1} Let $v$ be a positive solution to \eqref{eqno1.1}, $u=-\ln v$ and $h=|\nabla u|^2$ as before.
Then, there exists $\iota _0=c_{n,r,s}(1+\sqrt{\kappa}R)$, where $c_{n,r,s}=\max\{c_n,\,2\tilde\alpha,\,\frac{16}{\tilde\rho}\}$ is a positive constant depending on $n$, $\tilde\alpha$ and $\tilde\rho$ which is defined as in the above section, such that for any $0\leq\eta\in C_{0}^{\infty}(B_{R})$ and any $\iota \geq\iota _0$ large enough there holds true
\begin{equation*}
\begin{aligned}
&e^{-\iota _0}V^{\frac{2}{n}}\left(\int_{B_{R}}h^{(\iota +1)\chi }\eta^{2\chi }\right)^{\frac{1}{\chi }}+
4\iota \tilde\rho R^2\int_{B_{R}}h^{\iota +2}\eta^2\\
\leq&66R^2\int_{B_{R}}h^{\iota +1}\left|\nabla\eta\right|^2+\iota _0^2\iota  \int_{B_{R}} h^{\iota +1}\eta^2 .
\end{aligned}
\end{equation*}
Here $B_R$ is a geodesic Ball in $(M, g)$ and $V$ is the volume of $B_R$.
\end{lem}
\begin{proof}
Let
$$ A=\left\{x\in B_R|h(x)=0\right\},\quad \bar A=B_R\setminus A.$$
Thus, according to the Lemma \ref{lem2.1}, we take integration by part to derive that,
for any function $\varphi\in W_{0}^{1,2}(B_R)$ with $\varphi\geq 0$ and $\mbox{supp}(\varphi)\subset\subset\bar A$,
there holds true:
\begin{equation}\label{3.1}
\int_{B_{R}}\frac{\Delta \left(h^{\tilde\alpha}\right)}{\tilde\alpha h^{{\tilde\alpha}-1}}\varphi
\geq\tilde\rho \int_{B_{R}}h^{2}\varphi-2(n-1)\kappa \int_{B_{R}}h\varphi +\frac{2\left(n-2\right)}{n-1}\int_{B_{R}}\langle\nabla u,\,\nabla h\rangle \varphi,
\end{equation}
where $\tilde\alpha\geq1$ and $\tilde\rho\geq0$ are two suitable positive constants chosen in the proof of Lemma \ref{lem2.1}.

From (\ref{eqno2.5}) we have
$$
\int_{B_{R}}\Delta h\varphi=\int_{B_{R}}\frac{\Delta \left(h^{\tilde\alpha}\right)}{\tilde\alpha h^{\tilde\alpha-1}}\varphi-(\tilde\alpha-1)\int_{B_{R}}h^{-1}|\nabla h|^2\varphi.
$$
Substituting \eqref{3.1} into the above identity leads to
$$
\begin{aligned}
\int_{B_{R}}\Delta h\varphi
\geq&\tilde\rho \int_{B_{R}}h^{2}\varphi-(\tilde\alpha-1)\int_{B_{R}}h^{-1}|\nabla h|^2\varphi\\
&-2(n-1)\kappa \int_{B_{R}}h\varphi +\frac{2\left(n-2\right)}{n-1}\int_{B_{R}}\langle\nabla u,\,\nabla h\rangle \varphi.
\end{aligned}
$$
Hence, it follows
\begin{equation}\label{eqno3.2}
\begin{aligned}
\int_{B_{R}}\langle \nabla h,\, \nabla\varphi \rangle\
\leq&-\tilde\rho \int_{B_{R}}h^{2}\varphi+(\tilde\alpha-1)\int_{B_{R}}h^{-1}|\nabla h|^2\varphi\\
&+2(n-1)\kappa \int_{B_{R}}h\varphi -\frac{2\left(n-2\right)}{n-1}\int_{B_{R}}\langle\nabla u,\,\nabla h\rangle \varphi.
\end{aligned}
\end{equation}

Now, for any $\epsilon>0$ we define
$$h_\epsilon =\left(h-\epsilon\right)^+.$$
Let $\varphi=\eta^2 h_\epsilon^\iota \in W_{0}^{1,2}(B_R)$ where $0\leq\eta\in C_{0}^{\infty}(B_{R})$ and $\iota >\max\{1,\,2(\tilde\alpha-1)\}$ will be determined later.
Direct computation shows that
$$
\nabla \varphi= 2h_\epsilon^\iota \eta\nabla\eta+\iota  h_\epsilon^{\iota -1}\eta^2\nabla h.
$$
By substituting the above into (\ref{eqno3.2}), we derive
$$
\begin{aligned}
&\int_{B_{R}} \langle\nabla h,\,2h_\epsilon^\iota \eta\nabla\eta+\iota  h_\epsilon^{\iota -1}\eta^2\nabla h\rangle\\
\leq& -\tilde\rho \int_{B_{R}}h^2\eta^2 h_\epsilon^\iota+(\tilde\alpha-1)\int_{B_{R}}h^{-1}|\nabla h|^2\eta^2 h_\epsilon^\iota \\
&+2(n-1)\kappa\int_{B_{R}}h\eta^2 h_\epsilon^\iota  -\frac{2\left(n-2\right)}{n-1}\int_{B_{R}}\langle\nabla u,\,\nabla h\rangle  \eta^2 h_\epsilon^\iota ,
\end{aligned}
$$
it follows that
$$
\begin{aligned}
&2\int_{B_{R}} h_\epsilon^\iota   \eta\langle\nabla h,\,\nabla\eta\rangle+\iota \int_{B_{R}} h_\epsilon^{\iota -1} |\nabla h|^2\eta^2+\tilde\rho \int_{B_{R}}h^{2}\eta^2 h_\epsilon^\iota \\
\leq&(\tilde\alpha-1)\int_{B_{R}}h^{-1}|\nabla h|^2\eta^2 h_\epsilon^\iota+2(n-1)\kappa\int_{B_{R}}h\eta^2 h_\epsilon^\iota  -\frac{2\left(n-2\right)}{n-1}\int_{B_{R}}\langle\nabla u,\,\nabla h\rangle  \eta^2 h_\epsilon^\iota .
\end{aligned}
$$
Hence
$$
\begin{aligned}
&-2\int_{B_{R}} h_\epsilon^\iota   \eta|\nabla h||\nabla\eta|+\iota \int_{B_{R}} h_\epsilon^{\iota -1} |\nabla h|^2\eta^2+\tilde\rho \int_{B_{R}}h^{2}\eta^2 h_\epsilon^\iota\\
\leq&(\tilde\alpha-1)\int_{B_{R}}h^{-1}|\nabla h|^2\eta^2 h_\epsilon^\iota+ 2(n-1)\kappa\int_{B_{R}}h\eta^2 h_\epsilon^\iota  +\frac{2\left(n-2\right)}{n-1}\int_{B_{R}}|\nabla h| h^{\frac{1}{2}}\eta^2 h_\epsilon^\iota .
\end{aligned}
$$
By rearranging the above inequality, we have
$$
\begin{aligned}
\iota \int_{B_{R}} h_\epsilon^{\iota -1} |\nabla h|^2\eta^2+\tilde\rho \int_{B_{R}}h^{2}\eta^2 h_\epsilon^\iota
\leq& 2\int_{B_{R}} h_\epsilon^\iota   \eta|\nabla h||\nabla\eta|+(\tilde\alpha-1)\int_{B_{R}}h^{-1}|\nabla h|^2\eta^2 h_\epsilon^\iota\\
& +2(n-1)\kappa\int_{B_{R}}h\eta^2 h_\epsilon^\iota +\frac{2\left(n-2\right)}{n-1}\int_{B_{R}}|\nabla h| h^{\frac{1}{2}}\eta^2 h_\epsilon^\iota \\
\leq& 2\int_{B_{R}} h^{\iota }\eta|\nabla h||\nabla\eta|+(\tilde\alpha-1)\int_{B_{R}}h^{\iota -1}|\nabla h|^2\eta^2\\
&+2(n-1)\kappa\int_{B_{R}} h^{\iota +1}\eta^2 +\frac{2\left(n-2\right)}{n-1}\int_{B_{R}}|\nabla h| h^{\iota +\frac{1}{2}}\eta^2.
\end{aligned}
$$
By passing $\epsilon$ to $0$ we obtain
$$
\begin{aligned}
\iota \int_{B_{R}} h^{\iota -1}|\nabla h|^2\eta^2+\tilde\rho \int_{B_{R}}h^{\iota +2}\eta^2
\leq& 2\int_{B_{R}} h^{\iota }\eta|\nabla h||\nabla\eta|+(\tilde\alpha-1)\int_{B_{R}}h^{\iota -1}|\nabla h|^2\eta^2\\
&+2(n-1)\kappa\int_{B_{R}} h^{\iota +1}\eta^2 +\frac{2\left(n-2\right)}{n-1}\int_{B_{R}}|\nabla h| h^{\iota +\frac{1}{2}}\eta^2,
\end{aligned}
$$
then, by rearranging the above we have
$$
\begin{aligned}
&(\iota +1-\tilde\alpha)\int_{B_{R}} h^{\iota -1}|\nabla h|^2\eta^2+\tilde\rho \int_{B_{R}}h^{\iota +2}\eta^2\\
\leq&2\int_{B_{R}} h^{\iota }\eta|\nabla h||\nabla\eta|+2(n-1)\kappa\int_{B_{R}} h^{\iota +1}\eta^2 +\frac{2\left(n-2\right)}{n-1}\int_{B_{R}}|\nabla h| h^{\iota +\frac{1}{2}}\eta^2.
\end{aligned}
$$
Furthermore, by the choice of $\iota $ we know
\begin{equation}\label{eqno3.3}
\begin{aligned}
&\frac{\iota }{2}\int_{B_{R}} h^{\iota -1}|\nabla h|^2\eta^2+\tilde\rho \int_{B_{R}}h^{\iota +2}\eta^2\\
\leq& 2\int_{B_{R}} h^{\iota }\eta|\nabla h||\nabla\eta|+2(n-1)\kappa\int_{B_{R}} h^{\iota +1}\eta^2 +\frac{2\left(n-2\right)}{n-1}\int_{B_{R}}|\nabla h| h^{\iota +\frac{1}{2}}\eta^2.
\end{aligned}
\end{equation}

On the other hand, by Young's inequality we can derive
$$
\begin{aligned}
2\int_{B_{R}} h^{\iota }\eta|\nabla h||\nabla\eta|
=&2\int_{B_{R}}h^{\frac{\iota -1}{2}}\eta\left|\nabla h\right|\times h^{\frac{\iota +1}{2}}\left|\nabla\eta\right|\\
\leq&2\int_{B_{R}}\left[\frac{\iota }{8}\frac{h^{\iota -1}\eta^2\left|\nabla h\right|^2}{2} + \frac{8}{\iota } \frac{h^{\iota +1}\left|\nabla\eta\right|^2}{2}\right]\\
\leq&\frac{\iota }{8}\int_{B_{R}}h^{\iota -1}\eta^2\left|\nabla h\right|^2+\frac{8}{\iota }\int_{B_{R}}h^{\iota +1}\left|\nabla\eta\right|^2,
\end{aligned}
$$
and
$$
\begin{aligned}
\frac{2\left(n-2\right)}{n-1}\int_{B_{R}}|\nabla h| h^{\iota +\frac{1}{2}}\eta^2
\leq&2\int_{B_{R}}h^{\iota +\frac{1}{2}}\eta^2\left|\nabla h\right|\\
=&2\int_{B_{R}}h^{\frac{\iota -1}{2}}\eta\left|\nabla h\right|\times h^{\frac{\iota +2}{2}}\eta\\
\leq&2\int_{B_{R}}\left[\frac{\iota }{8}\frac{h^{\iota -1}\eta^2\left|\nabla h\right|^2}{2} + \frac{8}{\iota } \frac{h^{\iota +2}\eta^2}{2}\right]\\
\leq&\frac{\iota }{8}\int_{B_{R}}h^{\iota -1}\eta^2\left|\nabla h\right|^2+\frac{8}{\iota }\int_{B_{R}}h^{\iota +2}\eta^2.
\end{aligned}
$$

Now, by picking $\iota $ such that
$$
\iota \geq\max\{\frac{16}{\tilde\rho },\,2\tilde\alpha\}>\max\{1,\,2(\tilde\alpha-1)\})\quad\mbox{and}\quad\frac{8}{\iota }\leq\frac{\tilde\rho }{2},
$$
we can see easily that (\ref{eqno3.3}) can be rewritten as
\begin{equation}\label{eqno3.4}
\frac{\iota }{4}\int_{B_{R}} h^{\iota -1}|\nabla h|^2\eta^2+\frac{\tilde\rho }{2}\int_{B_{R}}h^{\iota +2}\eta^2
\leq 2(n-1)\kappa\int_{B_{R}} h^{\iota +1}\eta^2 +\frac{8}{\iota }\int_{B_{R}}h^{\iota +1}\left|\nabla\eta\right|^2.
\end{equation}

Besides, we have
$$
\begin{aligned}
\left|\nabla\left(h^{\frac{\iota +1 }{2}}\eta\right)\right|^2
=&\left|\eta\nabla h^{\frac{\iota +1}{2}}+h^{\frac{\iota +1}{2}}\nabla\eta\right|^2\\
\leq&2\eta^2\left|\nabla h^{\frac{\iota +1}{2}}\right|^2+2h^{\iota +1}\left|\nabla\eta\right|^2\\
=&\frac{\left(\iota +1\right)^2}{2}h^{\iota -1}\eta^2\left|\nabla h\right|^2+2h^{\iota +1}\left|\nabla\eta\right|^2,
\end{aligned}
$$
and integrate it on $B_{R}$ to obtain
$$
\begin{aligned}
\int_{B_{R}}\left|\nabla\left(h^{\frac{\iota +1 }{2}}\eta\right)\right|^2
\leq &\frac{\left(\iota +1\right)^2}{2}\int_{B_{R}}\eta^2h^{\iota -1 }\left|\nabla h\right|^2+2\int_{B_{R}}h^{\iota +1}\left|\nabla\eta\right|^2\\
\leq&\frac{2\left(\iota +1\right)^2}{\iota }\left[2(n-1)\kappa\int_{B_{R}} h^{\iota +1}\eta^2 +\frac{8}{\iota }\int_{B_{R}}h^{\iota +1}\left|\nabla\eta\right|^2 -\frac{\tilde\rho }{2}\int_{B_{R}}h^{\iota +2}\eta^2\right]\\
& +2\int_{B_{R}}h^{\iota +1}\left|\nabla\eta\right|^2.
\end{aligned}
$$
Noticing that there holds true $$\iota ^2<(\iota +1)^2\leq4\iota ^2,$$
we obtain
$$
\begin{aligned}
\int_{B_{R}}\left|\nabla\left(h^{\frac{\iota +1 }{2}}\eta\right)\right|^2
\leq&8\iota \left[2(n-1)\kappa\int_{B_{R}} h^{\iota +1}\eta^2+ \frac{8}{\iota }\int_{B_{R}}h^{\iota +1}\left|\nabla\eta\right|^2-\frac{\tilde\rho }{2}\int_{B_{R}}h^{\iota +2}\eta^2\right]\\
&+2\int_{B_{R}}h^{\iota +1}\left|\nabla\eta\right|^2\\
\leq&16(n-1)\kappa\iota \int_{B_{R}} h^{\iota +1}\eta^2 +66\int_{B_{R}}h^{\iota +1}\left|\nabla\eta\right|^2-4\iota \tilde\rho \int_{B_{R}}h^{\iota +2}\eta^2.
\end{aligned}
$$
According to the Theorem \ref{thm Sobolev}, we deduce from the above inequality
$$
\begin{aligned}
\left(\int_{B_{R}}h^{(\iota +1)\chi }\eta^{2\chi }\right)^{\frac{1}{\chi }}
\leq& e^{c_n\left(1+\sqrt \kappa R\right)}V^{-\frac{2}{n}}R^2\left[16(n-1)\kappa\iota \int_{B_{R}} h^{\iota +1}\eta^2 +66\int_{B_{R}}h^{\iota +1}\left|\nabla\eta\right|^2\right.\\
&\left.-4\iota \tilde\rho \int_{B_{R}}h^{\iota +2}\eta^2+R^{-2}\int_{B_{R}}h^{\iota +1}\eta^2\right]\\
=& e^{c_n\left(1+\sqrt \kappa R\right)}V^{-\frac{2}{n}}\left[\left(16(n-1)\kappa\iota  R^2+1\right)\int_{B_{R}} h^{\iota +1}\eta^2 \right.\\
&\left.+66R^2\int_{B_{R}}h^{\iota +1}\left|\nabla\eta\right|^2-4\iota \tilde\rho R^2\int_{B_{R}}h^{\iota +2}\eta^2\right],
\end{aligned}
$$
where $V=~\mbox{Vol}~(B_{R})$ and $\chi =\frac{n}{n-2}$. Rearranging the above inequality leads to the following
\begin{equation}\label{eqno3.5}
\begin{aligned}
&e^{-c_n\left(1+\sqrt \kappa R\right)}V^{\frac{2}{n}}\left(\int_{B_{R}}h^{(\iota +1)\chi }\eta^{2\chi }\right)^{\frac{1}{\chi }}
+4\iota \tilde\rho R^2\int_{B_{R}}h^{\iota +2}\eta^2\\
\leq&16(n-1)\iota  \left(\kappa R^2+1\right)\int_{B_{R}} h^{\iota +1}\eta^2 +66R^2\int_{B_{R}}h^{\iota +1}\left|\nabla\eta\right|^2.
\end{aligned}
\end{equation}

Now we choose
$$
\iota _0=c_{n,r,s}\left(1+\sqrt\kappa R\right),
$$
where $$c_{n,r,s}=\max\{c_n,\,2\tilde\alpha,\,\frac{16}{\tilde\rho }\}.
$$
Then, we can infer from \eqref{eqno3.5} that for any $\iota \geq\max\{\frac{16}{\tilde\rho},\, 2\tilde\alpha\}$ there holds
\begin{equation}\label{eqno3.6}
\begin{aligned}
&e^{-\iota _0}V^{\frac{2}{n}}\left(\int_{B_{R}}h^{(\iota +1)\chi }\eta^{2\chi }\right)^{\frac{1}{\chi }}+
4\iota \tilde\rho R^2\int_{B_{R}}h^{\iota +2}\eta^2\\
\leq&16(n-1)\iota  \left(\kappa R^2+1\right)\int_{B_{R}} h^{\iota +1}\eta^2 +66R^2\int_{B_{R}}h^{\iota +1}\left|\nabla\eta\right|^2.
\end{aligned}
\end{equation}
By the definition of $\tilde\rho $ in the previous section, it is easy to see that
$$
\frac{16}{\tilde\rho }\geq8(n-1)\geq 8
$$
and
$$
16(n-1)\left(\kappa R^2+1\right)\leq\left[c_{n,r,s}\left(1+\sqrt\kappa R\right)\right]^2=\iota _0^2.
$$
We can see that (\ref{eqno3.6}) can be rewritten as
\begin{equation}\label{eqno3.7}
\begin{aligned}
&e^{-\iota _0}V^{\frac{2}{n}}\left(\int_{B_{R}}h^{(\iota +1)\chi }\eta^{2\chi }\right)^{\frac{1}{\chi }}+
4\iota \tilde\rho R^2\int_{B_{R}}h^{\iota +2}\eta^2\\
\leq&\iota _0^2\iota  \int_{B_{R}} h^{\iota +1}\eta^2 +66R^2\int_{B_{R}}h^{\iota +1}\left|\nabla\eta\right|^2,
\end{aligned}
\end{equation}
where $\iota \geq\max\{\frac{16}{\tilde\rho },\,2\tilde\alpha\}$. We complete the proof of Lemma \ref{lem3.1}. \end{proof}
Using the above inequality we will infer a local estimate of $h$ stated in the following lemma, which will play a key role in the proofs of the main theorems.
\begin{lem}\label{lem3.2}
Let $\iota _1=(\iota _0+1)\chi $. Then there exist a universal constant $c>0$ such that the following estimate of $\|h\|_{L^{\iota _1}\left(B_{3R/4}\right)}$ holds
\begin{equation}\label{eqno3.8}
\begin{aligned}
\|h\|_{L^{\iota _1}\left(B_{3R/4}\right)}
\leq& \frac{c\iota _0^2}{\tilde\rho R^2}V^{\frac{1}{\iota _1}},
\end{aligned}
\end{equation}
where $c$ is a universal constant.
\end{lem}

\begin{proof}
Since the inequality (\ref{eqno3.7}) holds true for any $\iota \geq\iota _0$, now, by letting $\iota =\iota _0$ in (\ref{eqno3.7}), we can derive
\begin{equation}\label{eqno3.9}
\begin{aligned}
&e^{-\iota _0}V^{\frac{2}{n}}\left(\int_{B_{R}}h^{(\iota _0+1)\chi }\eta^{2\chi }\right)^{\frac{1}{\chi }}+
4\iota _0\tilde\rho R^2\int_{B_{R}}h^{\iota _0+2}\eta^2\\
\leq&\iota _0^3 \int_{B_{R}} h^{\iota _0+1}\eta^2 +66R^2\int_{B_{R}}h^{\iota _0+1}\left|\nabla\eta\right|^2.
\end{aligned}
\end{equation}
For simplicity, we denote the first term on the $RHS$ of (\ref{eqno3.9}) by $R_{1}$ ($R_2,\,L_1,\,L_2$ are understood similarly).
Now, we focus on the $R_1$. Note that if
$$h\geq\frac{\iota _0^2}{2\tilde\rho R^2},$$
then
$$R_1\leq 2\iota _0\tilde\rho R^2\int_{B_{R}}h^{\iota _0+2}\eta^2=\frac{L_2}{2};$$
and if
$$h<\frac{\iota _0^2}{2\tilde\rho R^2},$$
then
$$R_1<V\iota _0^3\left(\frac{\iota _0^2}{\tilde\rho R^2}\right)^{\iota _0+1}.$$
Therefore,
\begin{equation}\label{eqno3.10}
R_1\leq\frac{L_2}{2}+V\iota _0^3\left(\frac{\iota _0^2}{\tilde\rho R^2}\right)^{\iota _0+1}.
\end{equation}

Next, we need to calculate the term $R_2$ by choosing some special $\eta$. Choose $\eta_0\in C_0^{\infty}\left(B_R\right)$ such that
$$
\begin{cases}
0\leq\eta_0\leq1,&\mbox{on}~B_{R},\\
\eta_0=1,&\mbox{on}~B_{3R/4},\\
\left|\nabla\eta_0\right|\leq\frac{8}{R}. &
\end{cases}
$$
Let $\eta=\eta_0^{\iota _0+2}$, then, direct computation yields
$$
\begin{aligned}
R^2\left|\nabla\eta\right|^2
=&R^2(\iota _0+2)^2\eta_0^{2(\iota _0+1)}|\nabla\eta_0|^2\\
\leq& 4\iota _0^2\eta_0^{2(\iota _0+1)}\times8^2\\
=&256\iota _0^2\eta^{\frac{2(\iota _0+1)}{\iota _0+2}}.
\end{aligned}
$$
It means that one can find a universal constant $c$, which is independent of any parameter, such that
$$
R_2\leq c\iota _0^2\int_{B_{R}}h^{\iota _0+1}\eta^{\frac{2(\iota _0+1)}{\iota _0+2}}.
$$
By H\"{o}lder inequality, we have
$$
\begin{aligned}
c\iota _0^2\int_{B_{R}}h^{\iota _0+1}\eta^{\frac{2(\iota _0+1)}{\iota _0+2}}
\leq& c\iota _0^2\left( \int_{B_{R}}h^{\iota _0+2}\eta^2\right)^{\frac{\iota _0+1}{\iota _0+2}}\left(\int_{B_{R}}1\right)^{\frac{1}{\iota _0+2}}\\
=& c\iota _0^2\left( \int_{B_{R}}h^{\iota _0+2}\eta^2\right)^{\frac{\iota _0+1}{\iota _0+2}}V^{\frac{1}{\iota _0+2}}.
\end{aligned}
$$
Furthermore, for any $t>0$, we use Young's inequality to obtain
$$
\begin{aligned}
&c\iota _0^2\left( \int_{B_{R}}h^{\iota _0+2}\eta^2\right)^{\frac{\iota _0+1}{\iota _0+2}}V^{\frac{1}{\iota _0+2}}\\
=&\left( \int_{B_{R}}h^{\iota _0+2}\eta^2\right)^{\frac{\iota _0+1}{\iota _0+2}}t\times\frac{c\iota _0^2}{t}V^{\frac{1}{\iota _0+2}}\\
\leq&\frac{\iota _0+1}{\iota _0+2}\left[\left( \int_{B_{R}}h^{\iota _0+2}\eta^2\right)^{\frac{\iota _0+1}{\iota _0+2}}t\right]^\frac{\iota _0+2}{\iota _0+1}
+\frac{1}{\iota _0+2}\left(\frac{ c\iota _0^2}{t}V^{\frac{1}{\iota _0+2}}\right)^{\iota _0+2}\\
=&\frac{\iota _0+1}{\iota _0+2}t^\frac{\iota _0+2}{\iota _0+1} \int_{B_{R}}h^{\iota _0+2}\eta^2
+\frac{1}{\iota _0+2}t^{-(\iota _0+2)}\left(c\iota _0^2\right)^{\iota _0+2}V.
\end{aligned}
$$
Letting
$$
t=\left[\frac{2(\iota _0+2)\iota _0\tilde\rho R^2}{(\iota +1)}\right]^{\frac{\iota _0+1}{\iota _0+2}},
$$
we can see that
$$
\frac{\iota _0+1}{\iota _0+2}t^\frac{\iota _0+2}{\iota _0+1} =2\iota _0\tilde\rho R^2
$$
and
$$
\frac{1}{\iota _0+2}t^{-(\iota _0+2)}
=\frac{1}{\iota _0+2}\left[\frac{(\iota _0+1)}{2(\iota _0+2)\iota _0\tilde\rho R^2}\right]^{\iota _0+1}\leq\left(\frac{1}{\iota _0\tilde\rho R^2}\right)^{\iota _0+1}.
$$
Immediately, it follows
$$
\begin{aligned}
&c\iota _0^2\left( \int_{B_{R}}h^{\iota _0+2}\eta^2\right)^{\frac{\iota _0+1}{\iota _0+2}}V^{\frac{1}{\iota _0+2}}\\
\leq&2\iota _0\tilde\rho R^2\int_{B_{R}}h^{\iota _0+2}\eta^2
+\left(\frac{1}{\iota _0\tilde\rho R^2}\right)^{\iota _0+1}\left(c\iota _0^2\right)^{\iota _0+2}V\\
=&\frac{L_2}{2}+ c^{\iota _0+2}V\frac{\iota _0^2}{\iota _0^{\iota _0+1}}\left(\frac{\iota _0^2}{\tilde\rho R^2}\right)^{\iota _0+1}.
\end{aligned}
$$
Hence, we obtain
\begin{equation}\label{eqno3.11}
R_2\leq \frac{L_2}{2}+ c^{\iota _0+2}V\left(\frac{\iota _0^2}{\tilde\rho R^2}\right)^{\iota _0+1}.
\end{equation}
Substituting (\ref{eqno3.10}) and (\ref{eqno3.11}) into (\ref{eqno3.9}), we obtain
$$
\begin{aligned}
e^{-\iota _0}V^{\frac{2}{n}}\left(\int_{B_{R}}h^{(\iota _0+1)\chi }\eta^{2\chi }\right)^{\frac{1}{\chi }}
\leq&\iota _0^3\left(\frac{\iota _0^2}{\tilde\rho R^2}\right)^{\iota _0+1}V
+ c^{\iota _0+2}V\left(\frac{\iota _0^2}{\tilde\rho R^2}\right)^{\iota _0+1}\\
=&\left(\iota _0^3+c^{\iota _0+2}\right)V\left(\frac{\iota _0^2}{\tilde\rho R^2}\right)^{\iota _0+1},
\end{aligned}
$$
which implies
$$
\begin{aligned}
\left(\int_{B_{R}}h^{(\iota _0+1)\chi }\eta^{2\chi }\right)^{\frac{1}{\chi }}
\leq&\left(\iota _0^3+c^{\iota _0+2}\right)e^{\iota _0}V^{1-\frac{2}{n}}\left(\frac{\iota _0^2}{\tilde\rho R^2}\right)^{\iota _0+1}.
\end{aligned}
$$
Thus, we arrive at
$$
\begin{aligned}
\|h\|_{L^{\iota _1}\left(B_{3R/4}\right)}
\leq&\left(\iota _0^3+c^{\iota _0+2}\right)^{\frac{1}{\iota _0+1}}e^{\frac{\iota _0}{\iota _0+1}}V^{\frac{1}{\iota _1}}\frac{\iota _0^2}{\tilde\rho R^2}\\
\leq&2\left(\iota _0^{\frac{3}{\iota _0}}+c^2\right)eV^{\frac{1}{\iota _1}}\frac{\iota _0^2}{\tilde\rho R^2}\\
\leq&2ec^2\left(\iota _0^{\frac{3}{\iota _0}}+1\right)V^{\frac{1}{\iota _1}}\frac{\iota _0^2}{\tilde\rho R^2}.
\end{aligned}
$$
Here we have used the fact for any two positive number $a$ and $b$ there holds true $$(a+b)^p\leq a^p + b^p$$ as $0<p<1$. Furthermore, by the properties of the function $y(x)=x^{\frac{3}{x}}$ on $(0,\,+\infty)$ we know that for any $\iota _0>0$
$$\iota _0^{\frac{3}{\iota _0}}+1\leq e^{\frac{3}{e}} + 1=\max_{x\in(0, +\infty)}y(x).$$
Hence, (\ref{eqno3.8}) follows immediately. Thus, the proof of Lemma \ref{lem3.1} is completed.
\end{proof}
Now, we are in the position to give the proof of Theorem \ref{thm1} in the case $n\geq3$ by applying the Nash-Moser iteration method.
\begin{proof} Assume $v$ is a smooth positive solution of (\ref{eqno2.1}) with  $r\leq s$ on a complete Riemannian manifold $(M,\,g)$ with Ricci curvature $Ric(M)\geq-(n-1)\kappa$. When
$$1<r<\frac{n+3}{n-1}\quad\quad ~~\mbox{or}~~\quad 1<s<\frac{n+3}{n-1},$$
by the above arguments on
$$
h=\left|\nabla u\right|^2,
$$
where $u=-\ln v$, now we go back to (\ref{eqno3.7}) and ignore the second term on its $LHS$ to obtain
$$
\begin{aligned}
e^{-\iota _0}V^{\frac{2}{n}}\left(\int_{B_{R}}h^{(\iota +1)\chi }\eta^{2\chi }\right)^{\frac{1}{\chi }}
\leq& \iota _0^2\iota  \int_{B_{R}} h^{\iota +1}\eta^2 +66R^2\int_{B_{R}}h^{\iota +1}\left|\nabla\eta\right|^2\\
\leq&c\int_{B_{R}} h^{\iota +1}\left(\iota _0^2\iota  \eta^2 +R^2\left|\nabla\eta\right|^2\right),
\end{aligned}
$$
which is equivalent to
\begin{equation}\label{eqno3.12}
\left(\int_{B_{R}}h^{(\iota +1)\chi }\eta^{2\chi }\right)^{\frac{1}{\chi }}
\leq ce^{\iota _0}V^{-\frac{2}{n}}\int_{B_{R}} h^{\iota +1}\left(\iota _0^2\iota \eta^2  +R^2\left|\nabla\eta\right|^2\right),
\end{equation}
where $c$ is a universal positive constant which does not depend on any parameter.

In consideration of the delicate requirements of $\iota $, we take an increasing sequence $\{\iota _k\}_{k=1}^{\infty}$ such that
$$
\iota _1=(\iota _0+1)\chi \quad\mbox{and}\quad\iota _{k+1}=\iota _{k}\chi ,\quad k=1,\,2,\,...,
$$
and a decreasing one $\{r_k\}_{k=1}^{\infty}$ such that
$$
r_k=\frac{R}{2}+\frac{R}{4^k},\quad k=1,\,2,\,...\,.
$$
Then, we may choose $\{\eta_k\}_{k=1}^{\infty}\subset C_0^{\infty}(B_R)$, such that
$$
\eta_k\in C_0^{\infty}(B_{r_k}),\quad \eta_k=1~\mbox{in}~B_{r_{k+1}}\quad\mbox{and}\quad\left|\nabla\eta_k\right|\leq\frac{4^{k+1}}{R}.
$$
By letting $\iota +1=\iota _k$ and $\eta=\eta_k$ in (\ref{eqno3.12}), we derive
$$
\begin{aligned}
\left(\int_{B_{R}}h^{\iota_k\chi }
\eta_k^{2\chi }\right)^{\frac{1}{\chi }}
\leq& ce^{\iota _0}V^{-\frac{2}{n}}\int_{B_{R}} h^{\iota _k}\left[\iota _0^2\iota _k\eta_k^2  +R^2\left|\nabla\eta_k\right|^2\right]\\
\leq& ce^{\iota _0}V^{-\frac{2}{n}}\int_{B_{R}} h^{\iota _k}\left[\iota _0^2\iota _k\eta_k^2  +R^2\left(\frac{4^{k+1}}{R}\right)^2\right]\\
\leq& ce^{\iota _0}V^{-\frac{2}{n}}\left(\iota _0^2\iota _k  +16^{k+1}\right)\int_{B_{r_{k}}} h^{\iota _k}\\
\leq& ce^{\iota _0}V^{-\frac{2}{n}}\left[\iota _0^2(\iota _0+1)\chi ^k  +16^{k+1}\right]\int_{B_{r_{k}}} h^{\iota _k}\\
\leq& c e^{\iota _0}V^{-\frac{2}{n}}\left(\iota _0^3 16^k  +16^{k}\right)\int_{B_{r_{k}}} h^{\iota _k}\\
\leq& c e^{\iota _0}V^{-\frac{2}{n}}\iota _0^3 16^k \int_{B_{r_{k}}} h^{\iota _k}.
\end{aligned}
$$
Thus,
$$
\begin{aligned}
\left(\int_{B_{r_{k+1}}}h^{\iota _{k+1}}\right)^{\frac{1}{\iota _{k+1}}}
\leq& \left(c e^{\iota _0}V^{-\frac{2}{n}}\iota _0^3\right)^{\frac{1}{\iota _k}} 16^{\frac{k}{\iota _k}}\left( \int_{B_{r_{k}}} h^{\iota _k}\right)^{\frac{1}{\iota _k}},\\
\end{aligned}
$$
and this means that
$$
\|h\|_{L^{\iota _{k+1}}\left(B_{r_{k+1}}\right)}\leq
\left(ce^{\iota _0}V^{-\frac{2}{n}}\iota _0^3\right)^{\frac{1}{\iota _{k}}}16^{\frac{k}{\iota _k}} \|h\|_{L^{\iota _{k}}\left(B_{r_{k}}\right)}.
$$
By iteration we have
\begin{equation}\label{eqno3.13}
\|h\|_{L^{\iota _{k+1}}\left(B_{r_{k+1}}\right)}\leq
\left(ce^{\iota _0}V^{-\frac{2}{n}}\iota _0^3\right)^{\sum_{i=1}^{k}\frac{1}{\iota _{i}}}16^{\sum_{i=1}^{k}\frac{i}{\iota _i}} \|h\|_{L^{\iota _{1}}\left(B_{3R/4}\right)}.
\end{equation}
In view of
$$
\begin{aligned}
\sum_{i=1}^{\infty}\frac{1}{\iota _{i}}
=&\frac{1}{\iota _0+1}\sum_{i=1}^{\infty}\frac{1}{\chi ^i}\\
=&\frac{1}{\iota _0+1}\lim\limits_{i\to+\infty}\frac{\frac{1}{\chi }(1-\frac{1}{\chi ^i})}{1-\frac{1}{\chi }}\\
=&\frac{n-2}{\iota _0+1}\lim\limits_{i\to+\infty}(1-\frac{1}{\chi ^i})\\
=&\frac{n-2}{\iota _0+1}\\
=&\frac{n}{2\iota _1}
\end{aligned}
$$
and
$$
\begin{aligned}
\sum_{i=1}^{\infty}\frac{i}{\iota _{i}}
=&\frac{1}{\iota _0+1}\sum_{i=1}^{\infty}\frac{i}{\chi ^i}\\
=&\frac{1}{\iota _0+1}\frac{1}{\chi -1}\sum_{i=1}^{\infty}\left[(\chi -1)\frac{i}{\chi ^i}\right]\\
=&\frac{1}{\iota _0+1}\frac{n-2}{2}\sum_{i=1}^{\infty}\left(\frac{i}{\chi ^{i-1}}-\frac{i}{\chi ^{i}}\right)\\
=&\frac{n}{2\iota _1}\left[1+\lim\limits_{i\to+\infty}\left(\frac{1}{\chi }\frac{1-\frac{1}{\chi ^{i-1}}}{1-\frac{1}{\chi }}-\frac{i}{\chi ^{i}}\right)\right]\\
=&\frac{n}{2\iota _1}\left(1+\frac{1}{\chi -1}\right)\\
=&\frac{n^2}{4\iota _1},
\end{aligned}
$$
by letting $k\rightarrow\infty$ in (\ref{eqno3.13}) we obtain the following
$$
\|h\|_{L^{\infty}\left(B_{R/2}\right)}\leq c(n)V^{-\frac{1}{\iota _1}} \|f\|_{L^{\iota _{1}}\left(B_{3R/4}\right)}.
$$
By Lemma \ref{lem3.1}, we conclude from the above inequality that
$$
\|h\|_{L^{\infty}\left(B_{R/2}\right)}\leq c(n) \frac{\iota _0^2}{\tilde\rho R^2}.
$$
The definition of $\iota _0$ tells us that it follows
$$
\begin{aligned}
\|h\|_{L^{\infty}\left(B_{R/2}\right)}
\leq& c\left(n,\,\tilde\rho \right)\frac{\left(1+\sqrt\kappa R\right)^2}{R^2}\\
=& c\left(n,\,r,\,s\right)\frac{\left(1+\sqrt\kappa R\right)^2}{R^2}.
\end{aligned}
$$
\end{proof}

\subsection{The case $n=2$}
Next, we focus on the positive solutions of \eqref{eqno1.1} defined on a 2-dimensional complete Riemannian manifold with $Ric\geq -\kappa$.
According to the Lemma \ref{lem2.1}, we have the following claim:
\begin{lem}\label{lem2d.1}
Let $h=|\nabla u|^2$ and $u=-\ln v$ where $v$ is a positive solution to \eqref{eqno1.1}. Assume that $\dim(M)=n=2$. If $r\leq s$ and
$$1<r<5\quad\quad ~~\mbox{or}~~\quad 1<s<5,$$
then, there exist $\tilde\alpha\in\left[1,\,+\infty\right)$ and $\tilde\rho\in\left(0,\,2\right]$ such that, at the point where $h\neq0$, there holds
$$
\frac{\Delta \left(h^{\tilde\alpha}\right)}{\tilde\alpha h^{\tilde\alpha-1}}\geq\tilde\rho h^2-2\kappa h,
$$
where $\tilde\alpha=\tilde\alpha(r,\,s)$ and $\tilde\rho=\tilde\rho(r,\,s)$ are depend on $r$ and $s$. \end{lem}

Besides, according to the Theorem \ref{thm Sobolev}, for $n=2$, by letting $n'=2m$, where $m\in\mathbb{N^*}$ and $m>1$, we get the following direct corollary:
\begin{cor}\label{Sobolev-2d}
Let $(M,\,g)$ be a 2-dimensional complete Riemannian manifold with $Ric\geq -\kappa$.
there exist a constant $c_{2m}$, depending only on $m$, such that for all $B\subset M$ we have
$$
\left(\int_{B}h^{2\chi_m }\right)^{\frac{1}{\chi_m }}\leq e^{c_{2m}\left(1+\sqrt \kappa R\right)}V^{-\frac{1}{m}}R^2\left(\int_{B}\left|\nabla h\right|^2+\int_{B}R^{-2}h^2\right),
\quad f\in C_{0}^{\infty}(B),
$$
where $R$ and $V$ are the radius and volume of $B$ , constant $\chi_m =\frac{m}{m-1}$.
\end{cor}

By following almost the same argument as in the case $n\geq3$, we can easily get the following Lemmas.

\begin{lem}\label{lem2d.3}
Let $v$ be a positive solution to \eqref{eqno1.1} defined on a 2-dimensional complete Riemannian manifold with $Ric\geq -\kappa$, $u=-\ln v$ and $h=|\nabla u|^2$ as before.
Then, there exists $\iota ' _0=c_{r,\,s}(1+\sqrt{\kappa}R)$, where $c_{r,\,s}=\max\{c_{4},\,4,\,2\left(\tilde\alpha-1\right)\}$ is a positive constant depending on $\tilde\alpha$ which is defined in the proof of lemma \ref{lem2.1}, such that for any $0\leq\eta\in C_{0}^{\infty}(B_{R})$ and any $\iota ' \geq\iota ' _0$ large enough there holds true
\begin{equation}\label{eqno 2d.1}
\begin{aligned}
&e^{-\iota ' _0}V^{\frac{1}{2}}\left(\int_{B_{R}}h^{2(\iota ' +1) }\eta^{4 }\right)^{\frac{1}{2}}+
8\iota ' \tilde\rho R^2\int_{B_{R}}h^{\iota ' +2}\eta^2\\
\leq&34R^2\int_{B_{R}}h^{\iota ' +1}\left|\nabla\eta\right|^2+{\iota '_0}^2 \iota '\int_{B_{R}} h^{\iota ' +1}\eta^2 .
\end{aligned}
\end{equation}
Here $B_R$ is a geodesic Ball in $(M, g)$ and $V$ is the volume of $B_R$.
\end{lem}

\begin{proof}
Similar to the argument for Lemma \ref{lem3.1}, according to the Lemma \ref{lem2d.1}, for any $\iota ' >\max\{1,\,2(\tilde\alpha-1)\}$, we get
\begin{equation}\label{eqno 2d.3}
\frac{\iota ' }{2}\int_{B_{R}} h^{\iota ' -1}|\nabla h|^2\eta^2+\tilde\rho \int_{B_{R}}h^{\iota ' +2}\eta^2
\leq 2\int_{B_{R}} h^{\iota ' }\eta|\nabla h||\nabla\eta|+2\kappa\int_{B_{R}} h^{\iota ' +1}\eta^2 .
\end{equation}

Besides, by Young's inequality we can derive
$$
\begin{aligned}
2\int_{B_{R}} h^{\iota ' }\eta|\nabla h||\nabla\eta|
=&2\int_{B_{R}}h^{\frac{\iota ' -1}{2}}\eta\left|\nabla h\right|\times h^{\frac{\iota ' +1}{2}}\left|\nabla\eta\right|\\
\leq&2\int_{B_{R}}\left[\frac{\iota ' }{4}\frac{h^{\iota ' -1}\eta^2\left|\nabla h\right|^2}{2} + \frac{4}{\iota ' } \frac{h^{\iota ' +1}\left|\nabla\eta\right|^2}{2}\right]\\
\leq&\frac{\iota ' }{4}\int_{B_{R}}h^{\iota ' -1}\eta^2\left|\nabla h\right|^2+\frac{4}{\iota ' }\int_{B_{R}}h^{\iota ' +1}\left|\nabla\eta\right|^2.
\end{aligned}
$$
Substituting the above identity into \eqref{eqno 2d.3} leads to
\begin{equation}\label{eqno 2d.4}
\frac{\iota ' }{4}\int_{B_{R}} h^{\iota ' -1}|\nabla h|^2\eta^2+\tilde\rho \int_{B_{R}}h^{\iota ' +2}\eta^2
\leq 2\kappa\int_{B_{R}} h^{\iota ' +1}\eta^2+\frac{4}{\iota ' }\int_{B_{R}}h^{\iota ' +1}\left|\nabla\eta\right|^2 .
\end{equation}

Besides, since
$$
\begin{aligned}
\left|\nabla\left(h^{\frac{\iota ' +1 }{2}}\eta\right)\right|^2
\leq\frac{\left(\iota ' +1\right)^2}{2}h^{\iota ' -1}\eta^2\left|\nabla h\right|^2+2h^{\iota ' +1}\left|\nabla\eta\right|^2,
\end{aligned}
$$
then
$$
\begin{aligned}
\int_{B_{R}}\left|\nabla\left(h^{\frac{\iota ' +1 }{2}}\eta\right)\right|^2
\leq &\frac{\left(\iota ' +1\right)^2}{2}\int_{B_{R}}\eta^2h^{\iota ' -1 }\left|\nabla h\right|^2+2\int_{B_{R}}h^{\iota ' +1}\left|\nabla\eta\right|^2\\
\leq&\frac{2\left(\iota ' +1\right)^2}{\iota ' }\left[2\kappa\int_{B_{R}} h^{\iota ' +1}\eta^2+\frac{4}{\iota ' }\int_{B_{R}}h^{\iota ' +1}\left|\nabla\eta\right|^2-\tilde\rho \int_{B_{R}}h^{\iota ' +2}\eta^2\right]\\
& +2\int_{B_{R}}h^{\iota ' +1}\left|\nabla\eta\right|^2.
\end{aligned}
$$
Noticing that there holds true $${\iota '} ^2<(\iota ' +1)^2\leq4{\iota '} ^2,$$
we obtain
$$
\begin{aligned}
\int_{B_{R}}\left|\nabla\left(h^{\frac{\iota ' +1 }{2}}\eta\right)\right|^2
\leq&8\iota ' \left[2\kappa\int_{B_{R}} h^{\iota ' +1}\eta^2+\frac{4}{\iota ' }\int_{B_{R}}h^{\iota ' +1}\left|\nabla\eta\right|^2-\tilde\rho
\int_{B_{R}}h^{\iota ' +2}\eta^2\right]+2\int_{B_{R}}h^{\iota ' +1}\left|\nabla\eta\right|^2\\
\leq&16\kappa\iota ' \int_{B_{R}} h^{\iota ' +1}\eta^2 +34\int_{B_{R}}h^{\iota ' +1}\left|\nabla\eta\right|^2-8\iota ' \tilde\rho \int_{B_{R}}h^{\iota ' +2}\eta^2.
\end{aligned}
$$

According to the Corollary \ref{Sobolev-2d}, we obtain
$$
\begin{aligned}
\left(\int_{B_{R}}h^{(\iota ' +1)\chi_m }\eta^{2\chi_m  }\right)^{\frac{1}{\chi_m  }}
\leq& e^{c_{2m}\left(1+\sqrt \kappa R\right)}V^{-\frac{1}{m}}R^2\left[16\kappa\iota ' \int_{B_{R}} h^{\iota ' +1}\eta^2
 +34\int_{B_{R}}h^{\iota ' +1}\left|\nabla\eta\right|^2\right.\\
&\left.-8\iota ' \tilde\rho \int_{B_{R}}h^{\iota ' +2}\eta^2+R^{-2}\int_{B_{R}}h^{\iota ' +1}\eta^2\right]\\
\leq& e^{c_{2m}\left(1+\sqrt \kappa R\right)}V^{-\frac{1}{m}}\left[16\iota '\left(\kappa R^2+1\right) \int_{B_{R}} h^{\iota ' +1}\eta^2
+34R^2\int_{B_{R}}h^{\iota ' +1}\left|\nabla\eta\right|^2 \right.\\
&\left.-8\iota '\tilde\rho R^2\int_{B_{R}}h^{\iota ' +2}\eta^2\right],
\end{aligned}
$$
where $V=~\mbox{Vol}~(B_{R})$, $m\in\mathbb{N^*}(m>1)$ and $\chi_m =\frac{m}{m-1}$. Rearranging the above inequality leads to the following
$$
\begin{aligned}
&e^{-c_{2m}\left(1+\sqrt \kappa R\right)}V^{\frac{1}{m}}\left(\int_{B_{R}}h^{(\iota ' +1)\chi_m }\eta^{2\chi_m }\right)^{\frac{1}{\chi_m }}
+8\iota ' \tilde\rho R^2\int_{B_{R}}h^{\iota ' +2}\eta^2\\
\leq&16\iota '\left(\kappa R^2+1\right) \int_{B_{R}} h^{\iota ' +1}\eta^2 +34R^2\int_{B_{R}}h^{\iota ' +1}\left|\nabla\eta\right|^2.
\end{aligned}
$$
By letting $m=2$ and $\chi_m =2$, we have
\begin{equation}\label{eqno 2d.5}
\begin{aligned}
&e^{-c_4\left(1+\sqrt \kappa R\right)}V^{\frac{1}{2}}\left(\int_{B_{R}}h^{2(\iota ' +1) }\eta^{4 }\right)^{\frac{1}{2 }}
+8\iota ' \tilde\rho R^2\int_{B_{R}}h^{\iota ' +2}\eta^2\\
\leq&16\iota '\left(\kappa R^2+1\right) \int_{B_{R}} h^{\iota ' +1}\eta^2 +34R^2\int_{B_{R}}h^{\iota ' +1}\left|\nabla\eta\right|^2.
\end{aligned}
\end{equation}
Now we choose
$$
\iota ' _0=c_{r,\,s}\left(1+\sqrt\kappa R\right),
$$
where $$c_{r,\,s}=\max\{c_{4},\,4,\,2\left(\tilde\alpha-1\right)\}.
$$
It is not difficult to see that
$$
16\left(\kappa R^2+1\right) \leq{ \iota ' _0}^2.
$$
Then, we can infer from \eqref{eqno 2d.5} that for any $\iota ' >\max\{1,\,2\left(\tilde\alpha-1\right)\}$, inequality \eqref{eqno 2d.1} holds true.
We complete the proof of Lemma \ref{lem2d.3}.
\end{proof}

\begin{lem}\label{lem2d.4}
Let $\iota '_1=2(\iota '_0+1) $. Then there exist a universal constant $c>0$ such that the following estimate of $\|h\|_{L^{\iota ' _1}\left(B_{3R/4}\right)}$ holds
\begin{equation}\label{eqno 2d.2}
\begin{aligned}
\|h\|_{L^{\iota '_1}\left(B_{3R/4}\right)}
\leq& \frac{c{\iota '_0}^2}{\tilde\rho R^2}V^{\frac{1}{\iota '_1}},
\end{aligned}
\end{equation}
where $c$ is a universal constant.
\end{lem}

To prove Lemma \ref{lem2d.4}, we just need to following almost the same argument with respect to the Lemma \ref{lem3.2}, and we omit the details here.
\medskip

Now, according to the Lemma \ref{lem2d.4}, we can use Moser iteration technique to deduce that Theorem \ref{thm1} when $n=2$. Thus, Theorem \ref{thm1} is proved.

\subsection{The Proof of Corollary \ref{cor1}}
Now, we turn to proving Corollary \ref{cor1}.

\begin{proof}
Let $(M,\,g)$ be a noncompact complete Riemannian manifold with nonnegative Ricci curvature and $dim(M)\geq2$.  We assume $v$ is a smooth and positive solution of (\ref{eqno1.1}). If $r< s $ and
$$1<r<\frac{n+3}{n-1}\quad\quad ~~\mbox{or}~~\quad 1<s<\frac{n+3}{n-1},$$
Theorem \ref{thm1} tells us that there holds for any $B_{R}\subset M$,
$$
\frac{\left|\nabla v\right|^2}{v^2}\leq\frac{c(n,\,r,\,s)}{R^2},\quad\mbox{on}~B_{R/2}.
$$
Letting $R\rightarrow\infty$ yields $\nabla v=0$. Therefore, $v$ is a positive constant on $M$.
Furthermore, since $r<s$, we have that except for $u=1$
\begin{equation}\label{neq}
\Delta v +v^{r}-v^{s}= v^{r}-v^{s}\neq0.
\end{equation}
This is a contradiction which means that $v$ could not be the solution to (\ref{eqno1.1}) except for $u\equiv 1$. Hence we know that (\ref{eqno1.1}) admits a unique positive solution $u\equiv 1$. Thus we complete the proof of Corollary \ref{cor1}.\end{proof}

\end{document}